\documentclass[a4paper,12pt,reqno,english]{amsart}
\usepackage[utf8]{inputenc}
\usepackage[T1]{fontenc}
\usepackage{babel}

\usepackage[bookmarks = true, colorlinks, citecolor = blue, urlcolor = blue]{hyperref}
\usepackage[capitalize, nameinlink]{cleveref}
\usepackage[titletoc, toc, title]{appendix}

\numberwithin{equation}{section}
\numberwithin{figure}{section}

\theoremstyle{plain}
\newtheorem{theorem}{Theorem}[section]

\theoremstyle{plain}
\newtheorem*{theorem*}{Theorem}

\theoremstyle{plain}
\newtheorem{proposition}[theorem]{Proposition}

\theoremstyle{plain}
\newtheorem{lemma}[theorem]{Lemma}

\theoremstyle{plain}
\newtheorem{corollary}[theorem]{Corollary}

\theoremstyle{definition}
\newtheorem{definition}[theorem]{Definition}

\theoremstyle{definition}
\newtheorem{notation}[theorem]{Notation}

\theoremstyle{definition}

\theoremstyle{remark}
\newtheorem{remark}[theorem]{Remark}

\addto\extrasenglish{
  
}

\newcommand{\Uqg}{U_q(\mathfrak{g})}
\newcommand{\Uqsl}{U_q(\mathfrak{sl}_N)}

\newcommand{\id}{\mathrm{id}}
\newcommand{\braidR}{\widehat{\mathsf{R}}}
\newcommand{\symm}[1]{\mathcal{S}_{#1}}
\newcommand{\rescal}{\lambda}
\newcommand{\twistedext}[1]{\Lambda_{q, #1}(H)}

\begin{document}

\title[On PBW-deformations of braided exterior algebras]{On PBW-deformations of braided exterior algebras}

\author{Marco Matassa}

\address{Vakgroep wiskunde, Vrije Universiteit Brussel (VUB), B-1050 Brussels, Belgium}

\email{marco.matassa@gmail.com, marco.matassa@vub.ac.be}

\thanks{(Partially) supported by the FWO grant G.0251.15N}

\begin{abstract}
We classify PBW-deformations of quadratic-constant type of certain quantizations of exterior algebras.
These correspond to the fundamental modules of quantum $\mathfrak{sl}_N$, their duals, and their direct sums.
We show that the first two cases do not admit any deformation, while in the third case we obtain an essentially unique algebra with good properties. We compare this algebra with other quantum Clifford algebras appearing in the literature.
\end{abstract}

\maketitle

\section*{Introduction}

The aim of this paper is to initiate a study of PBW-deformations of certain quantizations of exterior algebras.
Recall that a \emph{PBW-deformation} of a quadratic algebra is a filtered algebra such that its associated graded algebra coincides with the original quadratic one, see for instance \cite{pbw-deformation, quadratic-algebras}.
The name comes from the main example of the universal enveloping algebra $U(\mathfrak{g})$ of a Lie algebra $\mathfrak{g}$, whose associated graded algebra is the symmetric algebra $S(\mathfrak{g})$, as a consequence of the Poincaré–Birkhoff–Witt theorem. Other classical examples include Weyl algebras and Clifford algebras, which can be seen as PBW-deformations of symmetric algebras and exterior algebras, respectively.
Here we will focus on deformations of quadratic-constant type, which classically correspond to Clifford algebras.

The quadratic algebras which we will consider are quantizations of the exterior algebras $\Lambda(V)$, where $V$ is a $U(\mathfrak{g})$-module and $\mathfrak{g}$ is a complex simple Lie algebra.
It is well known that the enveloping algebras $U(\mathfrak{g})$ admit quantizations $\Uqg$ as Hopf algebras, called quantized enveloping algebras or Drinfeld-Jimbo algebras. As long as the parameter $q$ is not a root of unity, the representation theory of $\Uqg$ essentially parallels that of $U(\mathfrak{g})$, so that in particular we have a $\Uqg$-module corresponding to $V$. A construction that functorially associates to $V$ a quadratic $\Uqg$-module algebra $\Lambda_q(V)$ is given in \cite{bezw}, and goes under the name of \emph{braided exterior algebras}.
However in general these algebras do not have the same Hilbert series as their classical counterparts. When this happens they are called \emph{flat}, and the flat simple modules have been classified in \cite{zwi}.
In the case of semisimple modules we can also consider the possibility of taking appropriate twisted tensor products.

One motivation for studying deformations of braided exterior algebras comes from the non-commutative geometry program of Connes \cite{con-book}. In this theory the main objects of study are spectral triples, whose main ingredients are Dirac-type operators, which are classically defined using Clifford algebras.
Hence developing appropriate quantum notions would lead to an interesting interaction between this theory and that of quantum groups.
Dirac operators for a certain class of quantum homogeneous spaces, namely quantized irreducible flag manifolds, have been defined long ago in \cite{qflag}.
This definition only uses some rather general assumptions on what a quantum Clifford algebra should be.
However, a more concrete model is needed for a detailed analysis of these operators, for instance to determine their spectra.
This analysis was performed for the case of quantum projective spaces in \cite{dd-proj} and \cite{mat-proj}, where it was shown that these Dirac operators give rise to spectral triples.
The general strategy to prove such a result is to obtain some quantum version of the Parthasarathy formula, see the discussion in \cite{mat-square} and \cite{mat-lagr}.
The proof of this formula in the classical case crucially uses the commutation relations in the Clifford algebra.
With such a general result still lacking in the quantum setting, we hope that a better understanding of some structural properties of quantum Clifford algebras will lead to some progress in this direction.

In this paper we will study PBW-deformations of braided exterior algebras corresponding to certain modules of $\Uqsl$. These are the fundamental module $V$ and its dual $V^*$, which are known to be flat, and their direct sum, which we denote by $H$. First we will show that the algebras $\Lambda_q(V)$ and $\Lambda_q(V^*)$ do not admit any non-trivial PBW-deformations of quadratic-constant type. The direct sum case on the other hand is far more interesting, also because we have to face the problem that $\Lambda_q(H)$ is not flat. For this reason we will replace it by an appropriate twisted tensor product of the algebras $\Lambda_q(V)$ and $\Lambda_q(V^*)$, which gives a quadratic algebra with the same Hilbert series as $\Lambda(H)$.
This is defined in terms of the braiding in the category of $\Uqsl$-modules, up to an important rescaling.
The main result is that in this case we find a one-parameter family of PBW-deformations, which we denote by $\mathrm{Cl}_q(c)$ for $c \in \mathbb{K}^\times$. They all turn out to be isomorphic for different values of $c$, as well as $\Uqsl$-module algebras. They admit a presentation which is very close to the classical case, namely
\[
\mathrm{Cl}_q(c) \cong T(H) / \langle x \otimes y + \sigma(x \otimes y) - (x, y)_c : x, y \in H \rangle.
\]
Here $\sigma : H \otimes H \to H \otimes H$ satisfies the braid equation, while the bilinear form $(\cdot, \cdot)_c : H \otimes H \to \mathbb{K}$ is $\Uqsl$-invariant and satisfies the symmetry property $(\cdot, \cdot)_c \circ \sigma = (\cdot, \cdot)_c$.

The paper is organized as follows. In \cref{sec:notation} we fix our conventions for quantized enveloping algebras.
In \cref{sec:pbw-deformations} we recall the notions of PBW-deformations and twisted tensor products.
In \cref{sec:braidings} we recall the notion of braided exterior algebra, as well as determining the braidings for the modules of interest.
In \cref{sec:braid-equation} we recall some facts about the braid equation and symmetrization.
In \cref{sec:simple-module} we show that there are no PBW-deformations for the simple modules we consider.
In \cref{sec:semisimple} we discuss our approach to the semisimple case, in terms of appropriate twisted tensor products.
Next in \cref{sec:pbw-semisimple} we classify the PBW-deformations of these algebras.
Finally in \cref{sec:properties} we show some further properties of these algebras, as well as the connection with other notions of quantum Clifford algebras.

\section{Quantized enveloping algebras}
\label{sec:notation}

In this section we fix some notation for complex simple Lie algebras and quantized enveloping algebras.
Let $\mathfrak{g}$ be a finite-dimensional complex simple Lie algebra.
Denote by $r$ the rank of $\mathfrak{g}$, by $\{ \alpha_i \}_{i = 1}^r$ the simple roots and by $\{ \omega_i \}_{i = 1}^r$ the fundamental weights. Denote by $\{ a_{i j} \}_{i, j = 1}^r$ the Cartan matrix. We will only consider the simply-laced case here.

For quantized enveloping algebras we use the conventions of \cite{jantzen}.
Let $q \in \mathbb{C}$ and suppose it is not a root of unity.
The \emph{quantized universal enveloping algebra} $U_q(\mathfrak{g})$ is generated by the elements $\{E_i$, $F_i$, $K_i$, $K_i^{-1}\}_{i = 1}^r$ satisfying the relations
\begin{gather*}
K_i K_i^{-1} = K_i^{-1} K_i=1,\ \
K_i K_j = K_j K_i, \\
K_i E_j K_i^{-1} = q^{a_{ij}} E_j,\ \
K_i F_j K_i^{-1} = q^{-a_{ij}} F_j, \\
E_i F_j - F_j E_i = \delta_{ij} \frac{K_i - K_i^{-1}}{q - q^{-1}},
\end{gather*}
plus the quantum analogue of the Serre relations.
The Hopf algebra structure is defined by
\begin{gather*}
\Delta(K_i) = K_i \otimes K_i, \quad
\Delta(E_i) = E_i \otimes 1 + K_i \otimes E_i, \quad
\Delta(F_i) = F_i \otimes K_i^{-1} + 1 \otimes F_i, \\
S(K_i) = K_i^{-1}, \quad
S(E_i) = - K_i^{-1} E_i, \quad
S(F_i) = - F_i K_i, \quad
\varepsilon(K_i) = 1, \quad
\varepsilon(E_i) = \varepsilon(F_i) = 0.
\end{gather*}
For $q \in \mathbb{R}$, the \emph{compact real form} of $U_q(\mathfrak{g})$ is defined by
\[
K_{i}^{*} = K_{i}, \quad
E_{i}^{*} = K_{i} F_{i}, \quad
F_{i}^{*} = E_{i} K_{i}^{-1}.
\]
Given $\lambda = \sum_{i = 1}^r n_i \alpha_i$ we will write $K_\lambda = K_1^{n_1} \cdots K_r^{n_r}$.
Let $\rho$ be the half-sum of the positive roots of $\mathfrak{g}$. Then we have $S^2(X) = K_{2\rho}^{-1} X K_{2\rho}$ for any $X \in \Uqg$.

The finite-dimensional irreducible (Type 1) representations of $U_{q}(\mathfrak{g})$ are labelled by their highest weights $\Lambda$ as in the classical case. We denote these modules by $V(\Lambda)$. For $q \in \mathbb{R}$ there is a Hermitian inner product $(\cdot, \cdot)$ on $V(\Lambda)$, unique up to a positive scalar factor, which is compatible with the compact real form in the sense that
\[
(a v, w) = (v, a^{*} w), \quad
v, w \in V(\Lambda), \ a \in U_{q}(\mathfrak{g}).
\]
We also need the braiding on the category of \emph{Type 1 representations}.
Given two $U_{q}(\mathfrak{g})$-modules $V$ and $W$ of this type, the \emph{braiding} is defined by $\braidR_{V, W} = \tau \circ R_{V, W}$.
Here $\tau$ denotes the flip and $R_{V, W}$ is the specialization of the universal $R$-matrix to the modules $V$ and $W$. It turns out that the braiding $\braidR_{V, W} : V \otimes W \to W \otimes V$ is uniquely determined by the relation
\[
\braidR_{V, W} (v \otimes w) = q^{(\mathrm{wt}(v), \mathrm{wt}(w))} w \otimes v + \sum_i w_i \otimes v_i,
\]
where $\mathrm{wt}(w_i) > \mathrm{wt}(w)$ and $\mathrm{wt}(v_i) < \mathrm{wt}(v)$.
Here $<$ denotes the natural partial order on the set of weights.
To determine $\braidR_{V, W}$ we can start by fixing a highest weight vector and obtain the other values using the action of $U_{q}(\mathfrak{g})$, since $\braidR_{V, W}$ is a module map.

\section{PBW-deformations and twisted tensor products}
\label{sec:pbw-deformations}

In this section we recall the notion of PBW-deformation of a quadratic algebra, following \cite{pbw-deformation}, and that of twisted tensor product of algebras, following \cite{twisted-tensor}.

\subsection{PBW-deformation}

Let $V$ be a vector space over a field $\mathbb{K}$ and denote by $T(V)$ its tensor algebra. Fix a subspace $R \subset V \otimes V$ and denote by $\langle R \rangle$ the two-sided ideal generated by $R$ inside $T(V)$. The algebra obtained by quotienting by such an ideal, which we denote by $Q(V, R) := T(V) / \langle R \rangle$, is called a (homogeneous) \emph{quadratic algebra}. More generally, write $F^2(V) = \mathbb{K} \oplus V \oplus (V \otimes V)$ and fix a subspace $P \subset F^2(V)$. Then the algebra $Q(V, P) = T(V) / \langle P \rangle$ is called a \emph{non-homogeneous quadratic algebra}.

The algebra $U = Q(V, P)$ has a natural structure of a filtered algebra. Hence we have an associated graded algebra, which we denote by $\mathrm{gr} U$.
Consider the projection $\pi : F^2(V) \to V \otimes V$ on the second homogeneous component. Set $R = \pi(P)$ and consider the homogeneous quadratic algebra $A = Q(V, R)$.
We have a natural surjective map $p : A \to \mathrm{gr} U$.

\begin{definition}[\cite{pbw-deformation}]
With the notation as above, we say that $U = Q(V, P)$ is a \emph{PBW-deformation} of $A = Q(V, R)$ if the projection $p : A \to \mathrm{gr} U$ is an isomorphism.
\end{definition}

It is easy to derive two necessary conditions for $U$ to be a PBW-deformation of $A$, which are given in \cite[Lemma 0.4]{pbw-deformation}.
The first condition implies that $P \subset F^2(V)$ can be described in terms of two linear maps $\alpha : R \to V$ and $\beta : R \to \mathbb{K}$ as
\[
P = \{ x - \alpha(x) - \beta(x) : x \in R \}.
\]
The second condition can be written in terms of various relations between $\alpha$ and $\beta$. Here we will consider only the case of \emph{quadratic-constant deformations}, namely $\alpha = 0$.

\begin{lemma}[{\cite[Lemma 3.3]{pbw-deformation}}]
Suppose $\alpha = 0$. Then we must have $\beta \otimes \id - \id \otimes \beta = 0$ on the intersection $(R \otimes V) \cap (V \otimes R) \subset V^{\otimes 3}$.
\end{lemma}

In general these conditions are necessary but not sufficient. However, for the class of Koszul algebras, these conditions turn out to be sufficient, as stated in the next theorem.

\begin{theorem}[{\cite[Theorem 4.1]{pbw-deformation}}]
Let $U = Q(V, P)$ and let $A = Q(V, R)$ be its associated quadratic algebra with $R = \pi(P)$. Suppose $\alpha$ and $\beta$ satisfy the necessary conditions. Suppose furthermore that $A$ is a Koszul algebra. Then $U$ is a PBW-deformation of $A$.
\end{theorem}

There are many equivalent definitions of a Koszul algebra. For these definitions the reader can consult the appendix of \cite{pbw-deformation} or the book \cite{quadratic-algebras}.

\subsection{Twisted tensor products}

The idea of a twisted tensor product of two algebras is a very natural one, and for this reason it has been studied independently by many authors. Here we will follow the treatment of \cite{twisted-tensor}.
We assume all algebras to be unital and all homomorphisms to preserve the units.

\begin{definition}[{\cite[Definition 2.1]{twisted-tensor}}]
Let $A$ and $B$ be algebras over $\mathbb{K}$. A \emph{twisted tensor product} of $A$ and $B$ is an algebra $C$, together with two injective homomorphisms $i_A : A \to C$ and $i_B : B \to C$,  such that the canonical linear map $(i_A, i_B) : A \otimes_\mathbb{K} B \to C$ defined by $(i_A, i_B)(a \otimes b) = i_A(a) i_B(b)$ is a linear isomorphism.
\end{definition}

Twisted tensor products can be characterized in terms of twisting maps. Indeed, given any twisted tensor product $(C, i_A, i_B)$ of two algebras $A$ and $B$, there exists a twisting map $\tau$, as defined below, such that $C$ is isomorphic to $A \otimes_\tau B$ \cite[Proposition 2.7]{twisted-tensor}.

Given a $\mathbb{K}$-linear map $\tau : B \otimes A \to A \otimes B$, and denoting by $\mu_A$ and $\mu_B$ multiplication on $A$ and $B$, we can consider as a candidate for multiplication on $A \otimes B$ the map
\[
\mu_\tau := (\mu_A \otimes \mu_B) (\id_A \otimes \tau \otimes \id_B).
\]
The notion of twisting map guarantees that this is an associative multiplication. The vector space $A \otimes B$, together with the multiplication map $\mu_\tau$, will be denoted by $A \otimes_\tau B$.

\begin{definition}[{\cite[Proposition/Definition 2.3]{twisted-tensor}}]
Let $\tau : B \otimes A \to A \otimes B$ be a $\mathbb{K}$-linear map. Then it is called a \emph{twisting map} if for all $a \in A$ and $b \in B$ it satisfies the conditions $\tau(b \otimes 1) = 1 \otimes b$ and $\tau(1 \otimes a) = a \otimes 1$, and moreover
\[
\tau (\mu_B \otimes \mu_A) = \mu_\tau (\tau \otimes \tau) (\id_B \otimes \tau \otimes \id_A).
\]
The multiplication $\mu_\tau$ is associative if and only if $\tau$ is a twisting map.

If $A$ and $B$ are graded algebras, we say that $\tau$ is \emph{graded} if $\tau(B_j \otimes A_i) \subseteq A_i \otimes B_j$ for all $i, j$.
\end{definition}

We now summarize some properties related to twisted tensor products in the presence of additional structure, see for example \cite[Section 1]{pbw-products}.
Recall that if $H$ is a Hopf algebra, then an $H$-module algebra $A$ is an algebra which is an $H$-module and such that $h \cdot (a a^\prime) = (h_{(1)} \cdot a) (h_{(2)} \cdot a^\prime)$ and $h \cdot 1_A = \varepsilon(h) 1_A$ for all $h \in H$ and $a, a^\prime \in A$.

\begin{proposition}
\label{prop:twisted-properties}
1) Let $A$ and $B$ be $H$-module algebras. If $\tau$ is an $H$-module homomorphism then the twisted tensor product $A \otimes_\tau B$ is an $H$-module algebra.

2) Let $A$ and $B$ be Koszul algebras. If $\tau$ is graded then $A \otimes_\tau B$ is a Koszul algebra.
\end{proposition}

\section{Braided exterior algebras and braidings}
\label{sec:braidings}

We begin this section by recalling the notion of braided exterior algebra, as defined in \cite{bezw}.
Next we determine the braidings associated with certain simple modules of $\Uqsl$.

\subsection{Braided exterior algebras}

The notion of braided exterior algebra provides a quantum version of the classical exterior algebra.
Given a $\Uqg$-module $V$, we can functorially associate to it an algebra $\Lambda_q(V)$ as follows \cite{bezw}. First we define the subspace
\begin{equation}
\label{eq:s2-braided}
S_q^2 V := \bigoplus_{\lambda > 0} \ker(\braidR_{V, V} - \lambda \ \id) \subset V \otimes V.
\end{equation}
This space is the span of the positive eigenvectors of $\braidR_{V, V}$, hence it is a quantum analogue of the space of symmetric $2$-tensors in $V \otimes V$.
Then we define the \emph{braided exterior algebra} $\Lambda_q(V) := T(V) / \langle S_q^2 V \rangle$. Braided symmetric algebras are also defined in a similar fashion.
By definition we have that $\Lambda_q(V)$ is a quadratic algebra. Moreover, since the braiding $\braidR_{V, V}$ is a module map, it follows that $\Lambda_q(V)$ is a $\Uqg$-module algebra.

While the definition makes sense for any $\Uqg$-module $V$, only for some particular ones we have that the Hilbert series of $\Lambda_q(V)$ and of the exterior algebra $\Lambda(V)$ coincide. When this happens we say that $V$ is a  \emph{flat} module. In this case it is known that $\Lambda_q(V)$ is a Koszul algebra.
Flat simple modules have been completely classified in \cite{zwi}.
The outcome of the classification is that the flat modules essentially correspond to abelian nilradicals of parabolic subalgebras of $\mathfrak{g}$, see \cite[Main Theorem 5.6]{zwi}.
Geometrically these can be interpreted as tangent spaces of the corresponding generalized flag manifolds.

\subsection{Braidings}

In the following we will be interested in studying various algebras associated to the fundamental module of $\Uqsl$, which we denote by $V := V(\omega_1)$, and its dual. The first step will be to determine the braiding associated with these modules.

The module $V(\omega_1)$ has a weight basis $\{ v_i \}_{i = 1}^N$, where the vector $v_i$ has weight $\lambda_i := \omega_i - \omega_{i - 1}$, with the convention that $\omega_0 = \omega_N = 0$. The $\Uqsl$-module structure can be realized by
\[
K_i v_j = q^{\delta_{i j} - \delta_{i, j - 1}} v_j, \quad
E_i v_j = \delta_{i, j - 1} v_{j - 1}, \quad
F_i v_j = \delta_{i j} v_{j + 1}.
\]
The braiding for $V \otimes V$ is well-known and can be found for example in \cite[Section 8.4.2]{klsc} (taking into account the rescaling by $q^{1 / N}$). Here and in the following we will write $\theta(n)$ for the Heaviside step function defined by $\theta(n) = 0$ for $n \leq 0$ and $\theta(n) = 1$ for $n > 0$.

\begin{proposition}
Let $V = V(\omega_1)$. Then we have
\[
\braidR_{V, V} (v_i \otimes v_j) = q^{\delta_{i j} - \frac{1}{N}} v_j \otimes v_i + \theta(j - i) q^{- \frac{1}{N}} (q - q^{-1}) v_i \otimes v_j.
\]
\end{proposition}

Next we will need the braiding for $V^* \otimes V^*$. It is very similar to that of $V \otimes V$. First we fix a basis for $V^*$ as follows.
Since $V$ is a simple module, there exists a unique (up to a constant) $\Uqsl$-invariant pairing $\langle \cdot, \cdot \rangle : V^* \otimes V \to \mathbb{C}$. Then we denote by $\{w_i\}_{i = 1}^N$ the dual basis, that is $\langle w_i, v_j \rangle = \delta_{i j}$.
It is easy to see that the $\Uqsl$-module structure is given by
\[
K_i w_j = q^{\delta_{i, j - 1} - \delta_{i j}} w_j, \quad
E_i w_j = - \delta_{i j} q^{-1} w_{j + 1}, \quad
F_i w_j = - \delta_{i, j - 1} q w_{j - 1}.
\]

\begin{proposition}
\label{prop:braiding-vstar}
Let $V^* \cong V(\omega_{N - 1})$. Then we have
\[
\braidR_{V^*, V^*} (w_i \otimes w_j) = q^{\delta_{i j} - \frac{1}{N}} w_j \otimes w_i + \theta(i - j) q^{- \frac{1}{N}} (q - q^{-1}) w_i \otimes w_j.
\]
\end{proposition}

We will also need the braiding for $V^* \otimes V \to V \otimes V^*$, which we determine below.

\begin{proposition}
\label{prop:braiding-vstar-v}
The braiding $\braidR_{V^*, V} : V^* \otimes V \to V \otimes V^*$ is given by
\[
\braidR_{V^*, V} (w_i \otimes v_j) = q^{\frac{1}{N} - \delta_{i j}} v_j \otimes w_i - \delta_{i j} q^{\frac{1}{N}} (q - q^{-1}) \sum_{k = 1}^{i - 1} v_k \otimes w_k.
\]
\end{proposition}

\begin{proof}
Observe that if $\braidR_{V^*, V} (w_i \otimes v_j)$ contains the term $v_k \otimes w_\ell$ then it must have the same weight as $w_i \otimes v_j$.
Since the vector $w_i$ has weight $-\lambda_i = \omega_{i - 1} - \omega_i$ this leads to the condition $- \lambda_i + \lambda_j = \lambda_k - \lambda_\ell$. More explicitly this reads
\[
\omega_{i - 1} - \omega_i + \omega_j - \omega_{j - 1} = \omega_k - \omega_{k - 1} + \omega_{\ell - 1} - \omega_\ell.
\]
Also from the properties of the braiding we have $\mathrm{wt}(v_k) \geq \mathrm{wt}(v_j)$ and $\mathrm{wt}(w_\ell) \leq \mathrm{wt}(w_i)$.
This is equivalent to $k \leq j$ and $\ell \leq i$, since $v_1$ and $w_N$ are the highest weight vectors of $V$ and $V^*$.

Now consider the case $i < j$. Then, taking into account the facts listed above, we conclude that $v_k \otimes w_\ell$ appears if and only if $(k, \ell) = (j, i)$. Similarly for $i > j$.
From this we conclude that the braiding takes the form $\braidR_{V^*, V} (w_i \otimes v_j) = q^{\frac{1}{N}} v_j \otimes w_i$ for $i \neq j$.

Now consider the case $i = j$. For $i = 1$ the result is true, since $v_1$ is the highest weight vector. We proceed by induction over $i$. We start by computing
\begin{align*}
\braidR_{V^*, V} F_i (w_{i + 1} \otimes v_i) & = \braidR_{V^*, V} (F_i w_{i + 1} \otimes K_i^{-1} v_i) + \braidR_{V^*, V} (w_{i + 1} \otimes F_i v_i) \\
& = - \braidR_{V^*, V} (w_i \otimes v_i) + \braidR_{V^*, V} (w_{i + 1} \otimes v_{i + 1}).
\end{align*}
On the other hand we have
\begin{align*}
F_i \braidR_{V^*, V} (w_{i + 1} \otimes v_i) & = q^{\frac{1}{N}} F_i v_i \otimes K_i^{-1} w_{i + 1} + q^{\frac{1}{N}} v_i \otimes F_i w_{i + 1} \\
& = q^{\frac{1}{N} - 1} v_{i + 1} \otimes w_{i + 1} - q^{\frac{1}{N} + 1} v_i \otimes w_i.
\end{align*}
Since $\braidR_{V^*, V}$ is a $\Uqsl$-module map, the two expressions coincide. Hence we get
\[
\braidR_{V^*, V} (w_{i + 1} \otimes v_{i + 1}) = q^{\frac{1}{N} - 1} v_{i + 1} \otimes w_{i + 1} - q^{\frac{1}{N} + 1} v_i \otimes w_i + \braidR_{V^*, V} (w_i \otimes v_i).
\]
Now plugging in the induction hypothesis and simplifying we obtain the result.
\end{proof}

\subsection{Some useful facts}

In this subsection we collect various useful facts about the modules $V$ and $V^*$ and their braidings.
We start by giving a vector space isomorphism between $V$ and $V^*$ which intertwines their respective braidings.

\begin{lemma}
\label{lem:iso-v-vstar}
We have a vector space isomorphism $\psi : V \to V^*$, defined by $\psi(v_i) = w_{N + 1 - i}$, such that $\braidR_{V^*, V^*} (\psi \otimes \psi) = (\psi \otimes \psi) \braidR_{V, V}$.
\end{lemma}

\begin{proof}
It is clear that it is an isomorphism, so we only have to check that it intertwines the braidings.
We apply $\braidR_{V^*, V^*}$ to $(\psi \otimes \psi) (v_i \otimes v_j) = w_{N + 1 - i} \otimes w_{N + 1 - j}$. Under the replacement $(i, j) \to (N + 1 - i, N + 1 - j)$ we have $\delta_{i j} \to \delta_{i j}$ and $\theta(i - j) \to \theta(j - i)$. Using these identities in the expression given in \cref{prop:braiding-vstar} we obtain
\begin{align*}
\braidR_{V^*, V^*} (\psi \otimes \psi) (v_i \otimes v_j) & = q^{\delta_{i j} - \frac{1}{N}} \psi(v_j) \otimes \psi(v_i) +  \theta(j - i) q^{- \frac{1}{N}} (q - q^{-1}) \psi(v_i) \otimes \psi(v_j) \\
& = (\psi \otimes \psi) \braidR_{V, V} (v_i \otimes v_j). \qedhere
\end{align*}
\end{proof}

\begin{remark}
This isomorphism does not respect the action of $\Uqsl$, since the simple modules $V = V(\omega_1)$ and $V^* \cong V(\omega_{N - 1})$ are not isomorphic.
\end{remark}

It is well known that $\braidR_{V, V}$ and $\braidR_{V^*, V^*}$ satisfy Hecke-type relations.

\begin{proposition}
\label{prop:hecke-relations}
The braidings $\braidR_{V, V}$ and $\braidR_{V^*, V^*}$ satisfy the Hecke-type relations
\[
(\braidR_{V, V} - q^{1 - \frac{1}{N}}) (\braidR_{V, V} + q^{-1 - \frac{1}{N}}) = 0, \quad
(\braidR_{V^*, V^*} - q^{1 - \frac{1}{N}}) (\braidR_{V^*, V^*} + q^{-1 - \frac{1}{N}}) = 0.
\]
\end{proposition}

\begin{proof}
It is immediate to check that $v_i \otimes v_j + q^{\delta_{i j} - 1} v_j \otimes v_i$ with $i \leq j$ and $v_i \otimes v_j - q v_j \otimes v_i$ with $i < j$ are eigenvectors with eigenvalues $q^{1 - \frac{1}{N}}$ and $- q^{- 1 - \frac{1}{N}}$, respectively. The case of $V^*$ follows by applying the map $\psi$ from \cref{lem:iso-v-vstar}, for example.
\end{proof}

Using these relations we can provide an equivalent description of the relations of the braided exterior algebras $\Lambda_q(V)$ and $\Lambda_q(V^*)$, which are given by $S_q^2 V$ and $S_q^2 V^*$ as in equation \eqref{eq:s2-braided}.

\begin{corollary}
\label{cor:image-hecke}
We have the identities
\[
S_q^2 V = \mathrm{im} (\id + q^{1 + \frac{1}{N}} \braidR_{V, V}), \quad
S_q^2 V^* = \mathrm{im} (\id + q^{1 + \frac{1}{N}} \braidR_{V^*, V^*}).
\]
\end{corollary}

\begin{proof}
Since $q^{1 - \frac{1}{N}}$ is the only positive eigenvalue of $\braidR_{V, V}$, we have $S_q^2 V = \ker(\braidR_{V, V} - q^{1 - \frac{1}{N}} \id)$. It follows from the quadratic relations of \cref{prop:hecke-relations} that
\[
\ker(\braidR_{V, V} - q^{1 - \frac{1}{N}} \id) = \mathrm{im} (\braidR_{V, V} + q^{-1 - \frac{1}{N}} \id).
\]
This can be rewritten as $S_q^2 V = \mathrm{im} (\id + q^{1 + \frac{1}{N}} \braidR_{V, V})$. Similarly for $V^*$.
\end{proof}

\section{Braid equation and symmetrization}
\label{sec:braid-equation}

In this section we will recall some facts related to solutions of the braid equation.
In particular we will introduce symmetrization operators, which will play an important role in checking the PBW-deformation condition later on.

\subsection{Braid equation and rescaling}

Let $V$ be a vector space and $\sigma : V \otimes V \to V \otimes V$ be a linear isomorphism. Write $\sigma_1 := \sigma \otimes \id$ and $\sigma_2 := \id \otimes \sigma$. Then we say that $\sigma$ satisfies the \emph{braid equation} if the equation $\sigma_1 \sigma_2 \sigma_1 = \sigma_2 \sigma_1 \sigma_2$ holds in $V^{\otimes 3}$. The pair $(V, \sigma)$ is called a braided vector space. The braid equation (or equivalently of the Yang-Baxter equation) plays an important role in the theory of Hopf algebras, see for example \cite[Section 8.1]{klsc}.

For reasons that will become clear later on, we will be interested in the following situation: given a solution of the braid equation on a direct sum $V = \bigoplus_{i \in I}V_i$, we want to consider a new map which is obtained from the given solution by rescaling its components by some constants. It is very easy to prove, as we will do below, that this rescaled map is again a solution of the braid equation.
Given any linear map $T : V \otimes V\to V \otimes V$ we will write $T = \sum_{i, j \in I} T_{i j}$ for its components, that is $T_{i j} : V_i \otimes V_j \to V_j \otimes V_i$.

\begin{proposition}
A map $T:V \otimes V \to V \otimes V$ satisfies the braid equation if and only if
\[
(T_{j k} \otimes \id) (\id \otimes T_{i k})(T_{i j} \otimes\id) = (\id \otimes T_{i j}) (T_{i k} \otimes\id) (\id \otimes T_{j k}), \quad \forall i, j, k \in I.
\]
\end{proposition}

\begin{proof}
Consider $T_1 T_2 T_1$. Then $V_i \otimes V_j \otimes V_k$ is mapped into
\[
V_i \otimes V_j \otimes V_k \xrightarrow{T_{i j} \otimes \id}  V_j \otimes V_i \otimes V_k \xrightarrow{\id \otimes T_{i k}}  V_j \otimes V_k \otimes V_i \xrightarrow{T_{j k} \otimes \id} V_k \otimes V_j \otimes V_i.
\]
Similarly consider $T_2 T_1 T_2$. Then $V_i \otimes V_j \otimes V_k$ is mapped into
\[
V_i \otimes V_j \otimes V_k \xrightarrow{\id \otimes T_{j k}} V_i \otimes V_k \otimes V_j \xrightarrow{T_{i k} \otimes \id} V_k \otimes V_i \otimes V_j \xrightarrow{\id \otimes T_{i j}} V_k \otimes V_j \otimes V_i.
\]
Comparing these two we obtain the stated conditions.
\end{proof}

\begin{corollary}
Suppose $T = \sum_{i, j \in I} T_{i j}$ satisfies the braid equation. Then $T^\prime = \sum_{i, j \in I} \lambda_{i j} T_{i j}$ satisfies the braid equation for any choice of scalars $\{\lambda_{i j}\}_{i, j \in I}$.
\end{corollary}

\begin{proof}
By the previous result, $T^\prime$ should satisfy the conditions
\[
\lambda_{j k} \lambda_{i k} \lambda_{i j} (T_{j k} \otimes \id) (\id \otimes T_{i k}) (T_{i j} \otimes \id) = \lambda_{i j} \lambda_{i k} \lambda_{j k} (\id \otimes T_{i j}) (T_{i k} \otimes\id) (\id \otimes T_{j k}).
\]
These are satisfied, since the prefactor is the same and $T$ satisfies the braid equation.
\end{proof}

\subsection{Symmetrization}

Now suppose $\sigma : V \otimes V \to V \otimes V$ satisfies the braid equation. Using $\sigma$ we can define an analogue of the symmetrization operator, similarly to the classical case with the flip map. In degree two and three this is given by
\[
\symm{2} := \id + \sigma, \quad
\symm{3} := \id + \sigma_1 + \sigma_2 + \sigma_1 \sigma_2 + \sigma_2 \sigma_1 + \sigma_1 \sigma_2 \sigma_1.
\]
Of course the fact that $\sigma$ satisfies the braid equation plays no role in the above definition. However, this property is needed to prove the following.

\begin{lemma}
\label{lem:symm-intersection}
Suppose $\sigma$ satisfies the braid equation. Then $\mathrm{im} \symm{3} \subset (\mathrm{im} \symm{2} \otimes V) \cap (V \otimes \mathrm{im} \symm{2})$.
\end{lemma}

\begin{proof}
To show that $\mathrm{im} \symm{3} \subset \mathrm{im} \symm{2} \otimes V$ we observe that
\begin{equation}
\label{eq:sym-first}
\symm{3} = (\id + \sigma_1) + (\id + \sigma_1) \sigma_2 + (\id + \sigma_1) \sigma_2 \sigma_1.
\end{equation}
Similarly to show that $\mathrm{im} \symm{3} \subset V \otimes \mathrm{im} \symm{2}$ we observe that
\begin{equation}
\label{eq:sym-second}
\symm{3} = (\id + \sigma_2) + (\id + \sigma_2) \sigma_1 + (\id + \sigma_2) \sigma_1 \sigma_2,
\end{equation}
where we have used the braid equation $\sigma_1 \sigma_2 \sigma_1 = \sigma_2 \sigma_1 \sigma_2$.
\end{proof}

Using the symmetrization map we can produce, in some cases, a convenient basis for checking the PBW-deformation condition. Indeed, this needs to be checked on the intersection $(R \otimes V) \cap (V \otimes R)$, where $R$ is the space of quadratic relations of the given quadratic algebra.

\section{Case simple modules}
\label{sec:simple-module}

In this section we will study PBW-deformations of quadratic-constant type for the braided exterior algebras $\Lambda_q(V(\omega_1))$ and $\Lambda_q(V(\omega_{N - 1}))$. We will show that they \emph{do not admit} PBW-deformations of this type. Hence there are no ``quantum Clifford algebras'' such that their associated graded algebras coincide with these braided exterior algebras.

Recall that quadratic-constant deformations are characterized by maps $\beta : R \to \mathbb{K}$, where $R \subset V \otimes V$ is the subspace of quadratic relations, such that $\beta \otimes \id - \id \otimes \beta = 0$ on the intersection $(R \otimes V) \cap (V \otimes R)$ inside $V^{\otimes 3}$.
In order to check this condition we will consider a convenient basis for this subspace, obtained via symmetrization.

In this section $\sigma$ will denote the maps $q^{1 + \frac{1}{N}} \braidR_{V, V}$ for $V(\omega_1)$ and $q^{1 + \frac{1}{N}} \braidR_{V^*, V^*}$ for $V(\omega_{N - 1})$. Then $\sigma$ satisfies the braid equation, as it is a given by a rescaling of the braiding.

\begin{notation}
\label{not:VandW}
We define $\mathcal{V}_{i j} := \symm{2}(v_i \otimes v_j)$ for $i \geq j$ and $\mathcal{W}_{i j} := \symm{2}(w_i \otimes w_j)$ for $i \leq j$.
\end{notation}

It follows from \cref{cor:image-hecke} that $\{ \mathcal{V}_{i j} \}_{i \geq j}$ and $\{ \mathcal{W}_{i j} \}_{i \leq j}$ are bases for the subspaces of quadratic relations $R_V$ and $R_{V^*}$, respectively. They are given explicitly by
\[
\mathcal{V}_{i j} = v_i \otimes v_j + q^{\delta_{i j} + 1} v_j \otimes v_i, \quad
\mathcal{W}_{i j} = w_i \otimes w_j + q^{\delta_{i j} + 1} w_j \otimes w_i.
\]
We introduce similar elements by symmetrization in degree three.

\begin{notation}
\label{not:VandW3}
We define $\mathcal{V}_{i j k} := \symm{3}(v_i \otimes v_j \otimes v_k)$ for $i \geq j \geq k$. Similarly we define $\mathcal{W}_{i j k} := \symm{3}(w_i \otimes w_j \otimes w_k)$ for $i \leq j \leq k$.
\end{notation}

In the next lemma we will show that
\[
\mathcal{V}_{i j k} \in (R_V \otimes V) \cap (V \otimes R_V), \quad
\mathcal{W}_{i j k} \in (R_{V^*} \otimes V^*) \cap (V^* \otimes R_{V^*}),
\]
and it will be clear that these elements are linearly independent. Then it follows that they are bases for these subspaces of $V^{\otimes 3}$ and $(V^*)^{\otimes 3}$. Indeed the algebras $\Lambda_q(V(\omega_1))$ and $\Lambda_q(V(\omega_{N - 1}))$ are flat, namely have the same Hilbert series as the corresponding classical exterior algebras, hence by dimensional reasons the elements above give a basis.

\begin{lemma}
\label{lem:vw-identities}
We have the identities
\begin{align*}
\mathcal{V}_{i j k} & = \mathcal{V}_{i j} \otimes v_k + q^{\delta_{j k} + 1} \mathcal{V}_{i k} \otimes v_j + q^{\delta_{i j} + \delta_{i k} + 2} \mathcal{V}_{j k} \otimes v_i \\
& = v_i \otimes \mathcal{V}_{j k} + q^{\delta_{i j} + 1} v_j \otimes \mathcal{V}_{i k} + q^{\delta_{j k} + \delta_{i k} + 2} v_k \otimes \mathcal{V}_{i j}.
\end{align*}
Similarly we have the identities
\begin{align*}
\mathcal{W}_{i j k} & = \mathcal{W}_{i j} \otimes w_k + q^{\delta_{j k} + 1} \mathcal{W}_{i k} \otimes w_j + q^{\delta_{i j} + \delta_{i k} + 2} \mathcal{W}_{j k} \otimes w_i \\
& = w_i \otimes \mathcal{W}_{j k} + q^{\delta_{i j} + 1} w_j \otimes \mathcal{W}_{i k} + q^{\delta_{j k} + \delta_{i k} + 2} w_k \otimes \mathcal{W}_{i j}.
\end{align*}
\end{lemma}

\begin{proof}
Using the fact that $i \geq j \geq k$ we compute
\[
\begin{gathered}
\sigma_1 (v_i \otimes v_j \otimes v_k) = q^{\delta_{i j} + 1} v_j \otimes v_i \otimes v_k, \quad
\sigma_2 (v_i \otimes v_j \otimes v_k) = q^{\delta_{j k} + 1} v_i \otimes v_k \otimes v_j, \\
\sigma_2 \sigma_1 (v_i \otimes v_j \otimes v_k) = q^{\delta_{i j} + \delta_{i k} + 2} v_j \otimes v_k \otimes v_i, \quad
\sigma_1 \sigma_2 (v_i \otimes v_j \otimes v_k) = q^{\delta_{j k} + \delta_{i k} + 2} v_k \otimes v_i \otimes v_j.
\end{gathered}
\]
Plugging these into the expressions for $\symm{3}$ given by \eqref{eq:sym-first} and \eqref{eq:sym-second} and using the definition of the elements $\mathcal{V}_{i j}$ we obtain the result.
To obtain the expressions for $\mathcal{W}_{i j k}$ we use the isomorphism $\psi : V \to V^*$ given in \cref{lem:iso-v-vstar}. Since $\psi$ intertwines the braidings we have
\begin{align*}
\mathcal{W}_{i j k} & = \symm{3} (\psi \otimes \psi \otimes \psi) (v_{N + 1 - i} \otimes v_{N + 1 - j} \otimes v_{N + 1 - k}) \\
& = (\psi \otimes \psi \otimes \psi) \mathcal{V}_{N + 1 - i, N + 1 - j, N + 1 - k}.
\end{align*}
Then applying $\psi$ to the expressions for $\mathcal{V}_{i j k}$ we obtain the result.
\end{proof}

We are now ready to study PBW-deformations of these braided exterior algebras.

\begin{theorem}
\label{thm:pbw-simple}
For $0 < q < 1$ there are no non-trivial PBW-deformations of quadratic-constant type of the algebras $\Lambda_q(V)$ and $\Lambda_q(V^*)$.
\end{theorem}

\begin{proof}
We will show that $(\beta \otimes \id - \id \otimes \beta) (\mathcal{V}_{i j k}) = 0$ implies $\beta(\mathcal{V}_{i j}) = 0$, namely a trivial deformation.
Applying $\beta \otimes \id - \id \otimes \beta$ to the two expressions given in \cref{lem:vw-identities} we get
\[
\begin{split}
0 & = \beta(\mathcal{V}_{i j}) v_k + q^{\delta_{j k} + 1} \beta(\mathcal{V}_{i k}) v_j + q^{\delta_{i j} + \delta_{i k} + 2} \beta(\mathcal{V}_{j k}) v_i \\
& - \beta(\mathcal{V}_{j k}) v_i - q^{\delta_{i j} + 1} \beta(\mathcal{V}_{i k}) v_j - q^{\delta_{j k} + \delta_{i k} + 2} \beta(\mathcal{V}_{i j}) v_k.
\end{split}
\]
The case $N = 2$ can be checked separately to show that $\beta(\mathcal{V}_{i j}) = 0$. Suppose $N > 2$ so that we can take $i, j, k$ to be all distinct. Then we obtain the condition
\[
(q^2 - 1) \beta(\mathcal{V}_{j k}) v_i + (1 - q^2) \beta(\mathcal{V}_{i j}) v_k = 0.
\]
Since $q^2 \neq 1$ this implies $\beta(\mathcal{V}_{i j}) = 0$ for all $i \neq j$.
Next consider the case $i = j$ and $j \neq k$. Using $\beta(\mathcal{V}_{a b}) = 0$ for $a \neq b$ we arrive at the condition $(1 - q^2) \beta(\mathcal{V}_{i i}) v_k = 0$, which implies $\beta(\mathcal{V}_{i i}) = 0$.
Finally the argument for $\Lambda_q(V^*)$ is completely identical.
\end{proof}

\section{Semisimple case}
\label{sec:semisimple}

Since the simple modules $V$ and $V^*$ do not admit PBW-deformations, we consider something with a bit more structure like the semisimple module $H := V(\omega_1) \oplus V(\omega_{N - 1})$.
However, as we will explain below, instead of considering the braided exterior algebra $\Lambda_q(H)$ we will use a different construction, which will occupy the rest of this section.

\subsection{A problem and a solution}

An unpleasant feature of the algebra $\Lambda_q(H)$ is its lack of flatness, despite the algebras $\Lambda_q(V)$ and $\Lambda_q(V^*)$ being flat deformations.
For example, for $N = 2$ we have $V \cong V^*$ and in \cite[Example 3.5.1.5]{tucker} it is shown that $\Lambda_q(V \oplus V)$ is not flat (actually this is shown for $S_q(V \oplus V)$, but the result can be immediately modified).
Hence we do not even have a vector space isomorphism between $\Lambda_q (H)$ and $\Lambda_q(V) \otimes \Lambda_q(V^*)$, which clashes with our expectations from the classical setting.

A similar problem was encountered in \cite{invariant-theory}, where the authors discuss a quantum version of the first fundamental theorem of classical invariant theory.
In the cited paper they observe that the quantum symmetric algebra $S_q(\oplus^m V)$, corresponding to $m$ copies of a certain module $V$, is not flat for $m > 1$.
Their solution is to replace $S_q(\oplus^m V)$ with a twisted tensor product of $m$ copies of $S_q(V)$, which is easily seen to be flat for all $m$.

Here we will use the same strategy and take the twisted tensor product of $\Lambda_q(V)$ and $\Lambda_q(V^*)$. However, unlike \cite{invariant-theory}, we will not simply use the braiding on the category of $\Uqsl$-modules as a twisting map, but we will consider a rescaled version of it. On one hand this is needed to introduce the appropriate minus signs for a tensor product of exterior algebras. On the other hand we will see that a non-trivial choice of this rescaling will be needed to obtain non-trivial PBW-deformations of quadratic-constant type.

\subsection{Relations of a twisted tensor product algebra}

The following lemma is fairly straightforward, and describes the space of relations of a twisted tensor product of two Koszul algebras, which is the setting we are interested in.

\begin{lemma}
\label{lem:relations-twisted}
Let $A = Q(V, R_V)$ and $B = Q(W, R_W)$ be Koszul algebras. Let $\tau : B \otimes A \to A \otimes B$ be a graded twisting map. Then the twisted tensor product $A \otimes_\tau B$ is isomorphic to the Koszul algebra $Q(V \oplus W, P)$, where $P = R_V \oplus R_W \oplus R_{V, W}$ and
\[
R_{V, W} = \{ w \otimes v - \tau(w \otimes v) : v \in V, w \in W \}.
\]
\end{lemma}

\begin{proof}
We have already recalled in \cref{prop:twisted-properties} that, given these assumptions, $A \otimes_\tau B$ is a Koszul algebra, hence we only need to determine the space of quadratic relations.
It is immediate, using the properties of a twisting map, to show the following identities
\begin{gather*}
(a \otimes 1) \cdot (a^\prime \otimes 1) = a a^\prime \otimes 1, \quad
(1 \otimes b) \cdot (1 \otimes b^\prime) = 1 \otimes b b^\prime, \\
(a \otimes 1) \cdot (1 \otimes b) = a \otimes b, \quad
(1 \otimes b) \cdot (a \otimes 1) = \tau(b \otimes a).
\end{gather*}
The first two identities show that $A$ and $B$ are subalgebras, hence we get the relations $R_V$ and $R_W$. Next, writing $\tau(b \otimes a) = \sum_i a_i \otimes b_i$ and using the last two identities, we get
\[
\sum_i (a_i \otimes 1) \cdot (1 \otimes b_i) = \sum_i a_i \otimes b_i = (1 \otimes b) \cdot (a \otimes 1).
\]
Since $\tau$ is a graded twisting map, that is $\tau(B_i \otimes A_j) \subset A_j \otimes B_i$, it suffices to impose this relation on the generators. 
Hence we get the subspace of relations
\[
R_{V, W} = \{ w \otimes v - \tau(w \otimes v) : v \in V, w \in W \} \subset (V \oplus W)^{\otimes 2}.
\]
By dimensional reasons these are all the relations.
\end{proof}

\subsection{Exterior algebra}

We will now define a quantum analogue of the exterior algebra $\Lambda(H)$, which will not coincide with the braided exterior algebra $\Lambda_q(H)$, as explained above.
Instead we will take an appropriate twisted tensor product of $\Lambda_q(V)$ and $\Lambda_q(V^*)$.

The twisting map we will use is built from the braiding of the category of $\Uqsl$-modules, which we have denoted by $\braidR$. Observe that $\braidR_{\Lambda_q(V^*), \Lambda_q(V)}$ is completely determined by $\braidR_{V^*, V}$ and by the multiplication maps $\wedge$, by naturality of the braiding.
Moreover it is easy to see that, if we rescale the braiding $\braidR_{V^*, V}$ by $\rescal \in \mathbb{K}^\times$, then there is a unique way to extend this to a graded twisting map. This motivates the following definition.

\begin{definition}
Let $\tau_\rescal : \Lambda_q(V^*) \otimes \Lambda_q(V) \to \Lambda_q(V) \otimes \Lambda_q(V^*)$ be the twisting map determined by $\tau_\rescal(w \otimes v) = - \rescal \braidR_{V^*, V}(w \otimes v)$, with $v \in V$, $w \in V^*$ and $\rescal \in \mathbb{K}^\times$. Then we define $\twistedext{\rescal} := \Lambda_q(V) \otimes_{\tau_\rescal} \Lambda_q(V^*)$ to be the twisted tensor product with respect to $\tau_\rescal$.
\end{definition}

Observe that $\tau_\rescal$ is a graded twisting map and a $\Uqsl$-module homomorphism. Then, since $\Lambda_q(V)$ and $\Lambda_q(V^*)$ are Koszul algebras, it follows from \cref{prop:twisted-properties} that $\twistedext{\rescal}$ is a $\Uqsl$-module algebra which is Koszul. We denote by $R_H \subset H \otimes H$ its subspace of quadratic relations.
From \cref{lem:relations-twisted} we have that $R_H = R_V \oplus R_{V^*} \oplus R_{V, V^*}$, where
\[
R_{V, V^*} = \{ w \otimes v + \rescal \braidR_{V^*, V} (w \otimes v) : v \in V, w \in V^* \}.
\]
These relations can be written in a more uniform way if we make the following definition.

\begin{notation}
We define the linear map $\sigma_\rescal : H \otimes H \to H \otimes H$ by
\[
\sigma_\rescal := \begin{cases}
q^{1 + \frac{1}{N}} \braidR_{V, V}, & V \otimes V \\
\lambda^{-1} \braidR_{V^*, V}^{-1}, & V \otimes V^* \\
\rescal \braidR_{V^*, V}, & V^* \otimes V\\
q^{1 + \frac{1}{N}}\braidR_{V^*, V^*}, & V^* \otimes V^*
\end{cases}.
\]
\end{notation}

We will sometimes omit the subscript $\rescal$ to avoid excessive clutter. With this definition we can write down the relations in a way which parallels the classical case.

\begin{proposition}
\label{prop:iso-image}
We have $\twistedext{\rescal} \cong T(H) / \langle R_H \rangle$, where $R_H = \mathrm{im} (\id + \sigma_\rescal)$.
\end{proposition}

\begin{proof}
We need to show that $R_H = R_V \oplus R_{V^*} \oplus R_{V, V^*}$  can be rewritten as $R_H = \mathrm{im} (\id + \sigma_\rescal)$.
It follows from \cref{cor:image-hecke} that we can write
\[
\mathrm{im} (\id + \sigma_\rescal) (V \otimes V) = R_V, \quad
\mathrm{im} (\id + \sigma_\rescal) (V^* \otimes V^*) = R_{V^*}.
\]
Next we have $(\id + \sigma_\rescal) (V^* \otimes V) = R_{V, V^*}$. Finally we show that $(\id + \sigma_\rescal) (V \otimes V^*) = R_{V, V^*}$.
Observe that $V \otimes V^* = \rescal \braidR_{V^*, V}(V^* \otimes V)$, since $\rescal \braidR_{V^*, V}$ is an isomorphism. Then
\[
(\id + \sigma_\rescal) (V \otimes V^*) = (\id + \lambda^{-1} \braidR_{V^*, V}^{-1}) \lambda \braidR_{V^*, V}(V^* \otimes V) = (\lambda \braidR_{V^*, V} + \id) (V^* \otimes V).
\]
From this computation we obtain the conclusion.
\end{proof}

\begin{remark}
The classical exterior algebra $\Lambda(H)$ is isomorphic to $T(H) / \langle \mathrm{im} (\id + \tau) \rangle$, where $\tau$ is the flip map $\tau(x \otimes y) = y \otimes x$. Clearly we have
\[
\mathrm{im} (\id + \tau) (V \otimes V^*) = \mathrm{im} (\id + \tau) (V^* \otimes V).
\]
Hence in the classical limit $q \to 1$ the algebra $\twistedext{\rescal}$ reduces to the exterior algebra $\Lambda(H)$ and the linear map $\sigma_\lambda$ reduces to the flip map $\tau$, provided that $\lambda \to 1$.
\end{remark}

The linear map $\sigma_\rescal$ satisfies the braid equation, since it is defined in terms of $\braidR$. Moreover we have $R_H = \mathrm{im} \symm{2}$, where $\symm{2}$ is the symmetrizer defined by $\sigma_\rescal$. Then it follows from \cref{lem:symm-intersection} that $\mathrm{im} \symm{3}$ is contained in the intersection $(R_H \otimes H) \cap (H \otimes R_H)$.
Using this fact we will obtain a basis of this subspace, which we will use to check the PBW-deformation condition.

\section{PBW-deformations in the semisimple case}
\label{sec:pbw-semisimple}

In this section we classify the PBW-deformations of quadratic-constant type for the algebras $\twistedext{\rescal}$. We start by giving a convenient basis for the intersection $(R_H \otimes H) \cap (H \otimes R_H) \subset H^{\otimes 3}$, which will be used to check the deformation condition for the linear map $\beta : R_H \to \mathbb{K}$.

\subsection{Basis elements}

We begin by writing down a basis for the subspace of quadratic relations $R_H = \mathrm{im} (\id + \sigma_\rescal) \subset H \otimes H$.
First observe that the elements $\mathcal{V}_{i j}$ and $\mathcal{W}_{i j}$ from \cref{not:VandW} can be written as $\mathcal{V}_{i j} = (\id + \sigma_\rescal) (v_i \otimes v_j)$ and $\mathcal{W}_{i j} = (\id + \sigma_\rescal) (w_i \otimes w_j)$.

\begin{notation}
\label{not:Mij}
We define $\mathcal{M}_{i j} := (\id + \sigma_\rescal) (w_i \otimes v_j$). We will also write $\rescal^\prime := \rescal q^{\frac{1}{N}}$.
\end{notation}

More explicitly, these elements are given by the expression
\begin{equation}
\label{eq:Mij-expression}
\mathcal{M}_{i j} = w_i \otimes v_j + \rescal^\prime q^{- \delta_{i j}} v_j \otimes w_i - \rescal^\prime \delta_{i j} (q - q^{-1}) \sum_{k = 1}^{i - 1} v_k \otimes w_k.
\end{equation}
It is clear that $\{ \mathcal{V}_{i j} \}_{i \geq j}$, $\{ \mathcal{W}_{i j} \}_{i \leq j}$ and $\{ \mathcal{M}_{i j} \}_{i, j}$ give a basis of $R_H$.

Next we need to determine a basis for the intersection $(R_H \otimes H) \cap (H \otimes R_H) \subset H^{\otimes 3}$. Clearly $\mathcal{V}_{i j k}$ and $\mathcal{W}_{i j k}$ from \cref{not:VandW3} belong to this intersection, since $R_V, R_{V^*} \subset R_H$.

\begin{notation}
We define the elements $\mathcal{X}_{i j k} := \symm{3}(w_i \otimes v_j \otimes v_k)$ for $j \geq k$.
\end{notation}

In the next lemma we obtain explicit expressions for $\mathcal{X}_{i j k}$.

\begin{lemma}
\label{lem:x-identities}
We have the identity
\[
\begin{split}
\mathcal{X}_{i j k} & = \rescal^{\prime 2} q^{- \delta_{i j} - \delta_{i k}} \mathcal{V}_{j k} \otimes w_i + q^{\delta_{j k} + 1} \mathcal{M}_{i k} \otimes v_j + \mathcal{M}_{i j} \otimes v_k \\
 & - \rescal^{\prime 2}(q - q^{-1}) \sum_{\ell = 1}^{i - 1} (\delta_{i k} q^{-\delta_{i j}} \mathcal{V}_{j \ell} + \delta_{i j} q^{-\delta_{\ell k}} \mathcal{V}_{\ell k}) \otimes w_{\ell} \\
 & + \delta_{i j} \theta(j - k)\rescal^{\prime 2} (q - q^{-1})^2 \sum_{\ell = 1}^{k - 1} \mathcal{V}_{k \ell}\otimes w_\ell.
\end{split}
\]
We also have the identity
\[
\begin{split}
\mathcal{X}_{i j k} & = w_i \otimes \mathcal{V}_{j k} + \rescal^\prime q^{-\delta_{i j}} v_j \otimes \mathcal{M}_{i k} + \rescal^\prime q^{-\delta_{i k} + \delta_{j k} + 1} v_k \otimes \mathcal{M}_{ij} \\
& - \rescal^\prime (q - q^{-1}) \sum_{\ell = 1}^{i - 1} v_\ell \otimes (\delta_{i j} \mathcal{M}_{\ell k} + \delta_{i k} q^{\delta_{j k} + 1} \mathcal{M}_{\ell j}).
\end{split}
\]
\end{lemma}

\begin{proof}
The two identities follow by applying the two expressions for $\symm{3}$ given by \eqref{eq:sym-first} and \eqref{eq:sym-second} to the elements $w_i \otimes v_j \otimes v_k$. First we compute
\[
\sigma_1 (w_i \otimes v_j \otimes v_k) = \rescal^\prime q^{-\delta_{i j}} v_j \otimes w_i \otimes v_k - \delta_{i j} \rescal^\prime (q - q^{-1}) \sum_{\ell = 1}^{i - 1} v_\ell \otimes w_\ell \otimes v_k.
\]
Next, since we have the condition $j \geq k$, we get $\sigma_2 (w_i \otimes v_j \otimes v_k) = q^{\delta_{j k} + 1} w_i \otimes v_k \otimes v_j$. Then
\[
\sigma_1 \sigma_2 (w_i \otimes v_j \otimes v_k) = \rescal^\prime q^{-\delta_{i k} + \delta_{j k} + 1} v_k \otimes w_i \otimes v_j - \delta_{i k} q^{\delta_{j k} + 1} \rescal^\prime (q - q^{-1}) \sum_{\ell = 1}^{i - 1} v_\ell \otimes w_\ell \otimes v_j.
\]
The most complicated term is the one obtained by applying $\sigma_2 \sigma_1$. We compute
\[
\begin{split}
\sigma_2 \sigma_1 (w_i \otimes v_j \otimes v_k) & = \sigma_2 \left( \rescal^\prime q^{-\delta_{i j}} v_j \otimes w_i \otimes v_k - \delta_{i j} \rescal^\prime (q - q^{-1}) \sum_{\ell = 1}^{i - 1} v_\ell \otimes w_\ell \otimes v_k \right) \\
 & = \rescal^\prime q^{-\delta_{i j}} v_j \otimes \left(\rescal^\prime q^{-\delta_{i k}} v_k \otimes w_i - \delta_{i k} \rescal^\prime (q - q^{-1}) \sum_{\ell = 1}^{i - 1} v_\ell \otimes w_\ell \right) \\
 & - \delta_{i j} \rescal^\prime (q - q^{-1}) \sum_{\ell = 1}^{i - 1} v_\ell \otimes \left(\rescal^\prime q^{-\delta_{\ell k}} v_k \otimes w_\ell - \delta_{\ell k} \rescal^\prime (q - q^{-1}) \sum_{m = 1}^{\ell - 1} v_m \otimes w_m \right).
\end{split}
\]
Since we have the condition $j \geq k$ we obtain the identity
\[
\delta_{i j} \sum_{\ell = 1}^{i - 1} \sum_{m = 1}^{\ell - 1} \delta_{\ell k} v_\ell \otimes v_m \otimes w_m = \delta_{i j} \theta(j - k) \sum_{m = 1}^{k - 1} v_k \otimes v_m \otimes w_m.
\]
Hence the expression for $\sigma_2 \sigma_1 (w_i \otimes v_j \otimes v_k)$ can be rewritten as
\[
\begin{split}
\sigma_2 \sigma_1 (w_i \otimes v_j \otimes v_k) & = \rescal^{\prime 2} q^{-(\delta_{i j} + \delta_{i k})} v_j \otimes v_k \otimes w_i \\
 & - \rescal^{\prime2} (q -  q^{-1}) \sum_{\ell = 1}^{i - 1} (\delta_{i j} q^{-\delta_{\ell k}} v_\ell \otimes v_k \otimes w_\ell + \delta_{i k} q^{-\delta_{i j}} v_j \otimes v_\ell  \otimes w_\ell) \\
 & + \delta_{i j} \theta(j - k)\rescal^{\prime 2} (q - q^{-1})^2 \sum_{\ell = 1}^{k - 1} v_k \otimes v_\ell \otimes w_\ell.
\end{split}
\]
Then the result follows by using the definitions of $\mathcal{V}_{i j}$ and $\mathcal{M}_{i j}$.
\end{proof}

\begin{notation}
We define the elements $\mathcal{Y}_{i j k} := \symm{3}(w_i \otimes w_j \otimes v_k)$ for $i \leq j$.
\end{notation}

In the next lemma we obtain explicit expressions for $\mathcal{Y}_{i j k}$.

\begin{lemma}
\label{lem:y-identities}
We have the identity
\[
\begin{split}
\mathcal{Y}_{i j k} & = \mathcal{W}_{i j} \otimes v_k + \rescal^\prime q^{- \delta_{j k}} \mathcal{M}_{i k} \otimes w_j + \rescal^\prime q^{\delta_{i j} - \delta_{i k} + 1} \mathcal{M}_{j k} \otimes w_i \\
& - \rescal^\prime (q - q^{-1}) \sum_{\ell = 1}^{k - 1} (\delta_{j k} \mathcal{M}_{i \ell} + \delta_{i k} q^{\delta_{i j} + 1} \mathcal{M}_{j \ell}) \otimes w_\ell.
\end{split}
\]
We also have the identity
\[
\begin{split}
\mathcal{Y}_{i j k} & = w_i \otimes \mathcal{M}_{j k} + q^{\delta_{i j} + 1} w_j \otimes \mathcal{M}_{i k} + \rescal^{\prime 2} q^{- \delta_{j k} - \delta_{i k}} v_k \otimes \mathcal{W}_{i j} \\
& - \rescal^{\prime 2} (q - q^{-1}) \sum_{\ell = 1}^{k - 1} v_\ell \otimes (\delta_{i k} q^{- \delta_{j k}} \mathcal{W}_{\ell j} + \delta_{j k} q^{- \delta_{i \ell}} \mathcal{W}_{i \ell}) \\
 & + \delta_{j k} \theta(j - i) \rescal^{\prime 2} (q - q^{-1})^2 \sum_{m = 1}^{i - 1} v_m \otimes \mathcal{W}_{m i}.
\end{split}
\]
\end{lemma}

\begin{proof}
Since $i \leq j$ we have $\sigma_1 (w_i \otimes w_j \otimes v_k) = q^{\delta_{i j} + 1} w_j \otimes w_i \otimes v_k$. Next we compute
\[
\sigma_2 (w_i \otimes w_j \otimes v_k) = \rescal^\prime q^{- \delta_{j k}} w_i \otimes v_k \otimes w_j - \delta_{j k} \rescal^\prime (q - q^{-1}) \sum_{\ell = 1}^{j - 1} w_i \otimes v_\ell \otimes w_\ell.
\]
Combining these two expressions we obtain
\[
\sigma_2 \sigma_1 (w_i \otimes w_j \otimes v_k) = \rescal^\prime q^{\delta_{i j} - \delta_{i k} + 1} w_j \otimes v_k \otimes w_i - \delta_{i k} \rescal^\prime q^{\delta_{i j} + 1} (q - q^{-1}) \sum_{\ell = 1}^{i - 1} w_j \otimes v_\ell \otimes w_\ell.
\]
The most complicated piece to compute is
\[
\begin{split}
\sigma_1 \sigma_2 (w_i \otimes w_j \otimes v_k) & =
\rescal^\prime q^{- \delta_{j k}} \sigma_1 (w_i \otimes v_k \otimes w_j) - \delta_{j k} \rescal^\prime (q - q^{-1}) \sum_{\ell = 1}^{j - 1} \sigma_1 (w_i \otimes v_\ell \otimes w_\ell) \\
 & = \rescal^\prime q^{- \delta_{j k}} \left( \rescal^\prime q^{- \delta_{i k}} v_k \otimes w_i - \delta_{i k} \rescal^\prime (q - q^{-1}) \sum_{\ell = 1}^{i - 1} v_\ell \otimes w_\ell \right) \otimes w_j \\
 & - \delta_{j k} \rescal^\prime (q - q^{-1}) \sum_{\ell = 1}^{j - 1} \left( \rescal^\prime q^{- \delta_{i \ell}} v_\ell \otimes w_i - \delta_{i \ell} \rescal^\prime (q - q^{-1}) \sum_{m = 1}^{i - 1} v_m \otimes w_m \right) \otimes w_\ell.
\end{split}
\]
Proceeding as in the previous lemma we find
\[
\begin{split}
\sigma_1 \sigma_2 (w_i \otimes w_j \otimes v_k)
& = \rescal^{\prime 2} q^{- \delta_{j k} - \delta_{i k}} v_k \otimes w_i \otimes w_j \\
& - \rescal^{\prime 2} (q - q^{-1}) \sum_{\ell = 1}^{k - 1} v_\ell \otimes (\delta_{i k} q^{- \delta_{j k}} w_\ell \otimes w_j + \delta_{j k} q^{- \delta_{i \ell}} w_i \otimes w_\ell) \\
& + \delta_{j k} \theta(j - i) \rescal^{\prime 2} (q - q^{-1})^2 \sum_{m = 1}^{i - 1} v_m \otimes w_m \otimes w_i.
\end{split}
\]
Plugging these into the two expressions for $\symm{3}$ we obtain the result.
\end{proof}

It is clear from their expressions that the elements $\{ \mathcal{X}_{i j k} \}_{i, j \geq k}$ and $\{ \mathcal{Y}_{i j k} \}_{i \leq j, k}$ are linearly independent. Then, together with $\{ \mathcal{V}_{i j k} \}_{i \geq j \geq k}$ and $\{ \mathcal{W}_{i j k} \}_{i \leq j \leq k}$, they give a basis of the intersection $(R_H \otimes H) \cap (H \otimes R_H)$. Indeed the algebra $\twistedext{\rescal}$ is a flat deformation of the classical exterior algebra $\Lambda(H)$, hence by dimensional reasons we have a basis.

\subsection{PBW-deformations}

We start by investigating the condition $(\beta \otimes \id - \id \otimes \beta) (\mathcal{X}_{i j k}) = 0$.

\begin{lemma}
\label{lem:beta-check1}
Suppose $(\beta \otimes \id - \id \otimes \beta) (\mathcal{X}_{i j k}) = 0$ for all $\mathcal{X}_{i j k}$.

1) If $\rescal^\prime \neq q$ then $\beta(\mathcal{M}_{i j}) = 0$ for all $i$ and $j$.

2) If $\rescal^\prime = q$ then $\beta(\mathcal{M}_{i j}) = \delta_{i j} c$ for some $c \in \mathbb{K}$.
\end{lemma}

\begin{proof}
To compute $(\beta \otimes \id - \id \otimes \beta) (\mathcal{X}_{i j k}) = 0$ we use the expressions for $\mathcal{X}_{i j k}$ given in \cref{lem:x-identities}. Taking into account that $\beta(\mathcal{V}_{a b}) = 0$ for all $a, b$ we obtain
\[
\begin{split}
0 & = q^{\delta_{j k} + 1} \beta(\mathcal{M}_{i k}) v_j + \beta(\mathcal{M}_{i j}) v_k \\
& - \rescal^\prime q^{-\delta_{i j}} \beta(\mathcal{M}_{i k}) v_j - \rescal^\prime q^{-\delta_{i k} + \delta_{j k} + 1} \beta(\mathcal{M}_{ij}) v_k \\
& + \rescal^\prime (q - q^{-1}) \sum_{\ell = 1}^{i - 1} (\delta_{i j} \beta(\mathcal{M}_{\ell k}) + \delta_{i k} q^{\delta_{j k} + 1} \beta(\mathcal{M}_{\ell j})) v_\ell.
\end{split}
\]
We omit the verification of the case $N = 2$, which can be checked separately. Suppose that $N > 2$, so that is possible to choose $i, j, k$ are all distinct. We get
\[
(q - \rescal^\prime) \beta(\mathcal{M}_{i k}) v_j + (1 - \rescal^\prime q) \beta(\mathcal{M}_{i j}) v_k = 0.
\]
This implies $\beta(\mathcal{M}_{i j}) = 0$ for $i \neq j$. Next consider the case $i = j = k$. We obtain
\[
0 = (q + q^{-1}) (q - \rescal^\prime) \beta(\mathcal{M}_{i i}) v_i,
\]
where we have used that $\beta(\mathcal{M}_{i j}) = 0$ for $i \neq j$.
Then we must have $\beta(\mathcal{M}_{i i}) = 0$ unless $\rescal^\prime = q$.

To check the remaining cases we fix $\rescal^\prime = q$. Let us rewrite the general condition as
\[
\begin{split}
0 & = q (q^{\delta_{j k}} - q^{-\delta_{i j}}) \beta(\mathcal{M}_{i k}) v_j + (1 - q^{-\delta_{i k} + \delta_{j k} + 2}) \beta(\mathcal{M}_{ij}) v_k \\
& + q (q - q^{-1}) \sum_{\ell = 1}^{i - 1} (\delta_{i j} \beta(\mathcal{M}_{\ell k}) + \delta_{i k} q^{\delta_{j k} + 1} \beta(\mathcal{M}_{\ell j})) v_\ell.
\end{split}
\]
We are left with checking the cases where exactly two indices coincide.
In the case $j = k$ and $i \neq j$ we see that the condition is identically satisfied. Next for $i = k$ and $i \neq j$ we have
\[
q (q - q^{-1}) \sum_{\ell = 1}^{i - 1} q \beta(\mathcal{M}_{\ell j}) v_\ell = 0.
\]
Recall that the elements $\mathcal{X}_{i j k}$ are defined for $j \geq k$. In this case this implies $j > i$ and hence the term above vanishes. Finally we are left with the case $i = j$ and $j \neq k$. We get
\[
- q (q - q^{-1}) \beta(\mathcal{M}_{ii}) v_k + q (q - q^{-1}) \sum_{\ell = 1}^{i - 1} \beta(\mathcal{M}_{\ell k}) v_\ell = 0.
\]
As above we have $i = j > k$. Hence this condition can be rewritten as
\[
- q (q - q^{-1}) \beta(\mathcal{M}_{ii}) v_k + q (q - q^{-1}) \beta(\mathcal{M}_{k k}) v_k = 0.
\]
Since $0 < q < 1$ we conclude that $\beta(\mathcal{M}_{i i}) = \beta(\mathcal{M}_{k k})$ for $i > k$.
\end{proof}

Next we check the PBW-deformation condition for the elements $\mathcal{Y}_{i j k}$.

\begin{lemma}
\label{lem:beta-check2}
Suppose $\rescal^\prime = q$ and $\beta(\mathcal{M}_{i j}) = \delta_{i j} c$. Then $(\beta \otimes \id - \id \otimes \beta) (\mathcal{Y}_{i j k}) = 0$ is satisfied.
\end{lemma}

\begin{proof}
We apply the linear map $\beta$ to the two expressions for $\mathcal{Y}_{i j k}$ given in \cref{lem:y-identities}.
First, using $\beta(\mathcal{W}_{i j}) = 0$ and $\beta(\mathcal{M}_{i j}) = \delta_{i j} c$, we compute
\[
c^{-1} (\beta \otimes \id) \mathcal{Y}_{i j k} = \delta_{i k} q^{- \delta_{j k} + 1} w_j + \delta_{j k} q^{\delta_{i j} - \delta_{i k} + 2} w_i - q (q - q^{-1}) \sum_{\ell = 1}^{k - 1} (\delta_{j k} \delta_{i \ell} + \delta_{i k} q^{\delta_{i j} + 1} \delta_{j \ell}) w_\ell.
\]
Since $i \leq j$ by definition of the elements $\mathcal{Y}_{i j k}$, we have the identity
\[
\sum_{\ell = 1}^{k - 1} (\delta_{j k} \delta_{i \ell} + \delta_{i k} q^{\delta_{i j} + 1} \delta_{j \ell}) w_\ell = \delta_{j k} \theta(j - i) w_i.
\]
Plugging this in and simplifying we obtain the expression
\[
c^{-1} (\beta \otimes \id) \mathcal{Y}_{i j k} = \delta_{i k} q^{- \delta_{j k} + 1} w_j + \delta_{j k} q^2 w_i - \delta_{j k} \theta(j - i) q (q - q^{-1}) w_i.
\]
But then a simple inspection shows that it is the same as the other term
\[
c^{-1} (\id \otimes \beta) \mathcal{Y}_{i j k} = \delta_{j k} w_i + \delta_{i k} q^{\delta_{i j} + 1} w_j.
\]
For example in the case $i \neq j$ and $j = k$ we get
\[
c^{-1} (\beta \otimes \id) \mathcal{Y}_{i j j} = q^2 w_i - q (q - q^{-1}) w_i = w_i.
\]
On the other hand we have $c^{-1} (\id \otimes \beta) \mathcal{Y}_{i j j} = w_i$.
\end{proof}

Now we are ready to classify the PBW-deformations of the algebra $\twistedext{\rescal}$.

\begin{theorem}
\label{thm:pbw-semisimple}
The PBW-deformations of $\twistedext{\rescal}$ of quadratic-constant type are as follows.

1) For $\lambda \neq q^{1 - \frac{1}{N}}$ there are no non-trivial deformations.

2) For $\lambda = q^{1 - \frac{1}{N}}$ the non-trivial deformations are parametrized by $c \in \mathbb{K}^\times$ and given by
\[
\beta(\mathcal{V}_{i j}) = \beta(\mathcal{W}_{i j}) = 0, \quad
\beta(\mathcal{M}_{i j}) = \delta_{i j} c.
\]
\end{theorem}

\begin{proof}
Recall that $\twistedext{\rescal}$ is a Koszul algebra, since it is the twisted tensor product of two Koszul algebras. Hence the condition $\beta \otimes \id - \id \otimes \beta = 0$ on $(R_H \otimes H) \cap (H \otimes R_H)$ is sufficient, as well as necessary.
A vector space basis for $(R_H \otimes H) \cap (H \otimes R_H)$ is given by the elements
\[
\{ \mathcal{V}_{i j k} \}_{i \geq j \geq k}, \quad
\{ \mathcal{W}_{i j k} \}_{i \leq j \leq k}, \quad
\{ \mathcal{X}_{i j k} \}_{i, j \geq k}, \quad
\{ \mathcal{Y}_{i j k} \}_{i \leq j, k}.
\]
We have seen in \cref{thm:pbw-simple} that $\beta(\mathcal{V}_{i j}) = \beta(\mathcal{W}_{i j}) = 0$. On the other hand the conditions for $\beta(\mathcal{M}_{i j})$ follow from \cref{lem:beta-check1} and \cref{lem:beta-check2}.
\end{proof}

\section{Further properties}
\label{sec:properties}

In this section we will explore some further properties of the PBW-deformations obtained in the previous section, which we will denote from now by $\mathrm{Cl}_q(c)$. We will obtain a presentation similar to the classical setting, prove that they are $\Uqsl$-module algebras and that they are isomorphic for different values of $c \in \mathbb{K}^\times$. Finally we will compare this quantum Clifford algebra to other notions of quantum Clifford algebras defined by other authors.

Recall that in this section we only consider the value $\lambda = q^{1 - \frac{1}{N}}$ and $\beta_c(\mathcal{M}_{i j}) = \delta_{i j} c$ for some $c \in \mathbb{K}^\times$. We will also write the linear map $\sigma_\rescal$ for this value simply as $\sigma$. Finally recall that we denote by $\langle \cdot, \cdot \rangle : V^* \otimes V \to \mathbb{C}$ the $\Uqsl$-invariant pairing such that $\langle w_i, v_j \rangle = \delta_{i j}$.

\begin{proposition}
\label{prop:cliff-presentation}
The algebra $\mathrm{Cl}_q(c)$ is isomorphic to
\[
T(H) / \langle x \otimes y + \sigma(x \otimes y) - (x, y)_c : x, y \in H \rangle,
\]
where the bilinear form $(\cdot, \cdot)_c : H \otimes H \to \mathbb{C}$ is defined by
\[
(\cdot, \cdot)_c := \begin{cases}
c \langle \cdot, \cdot \rangle, & V^* \otimes V\\
c \langle \cdot, \cdot \rangle \circ q^{-1 + \frac{1}{N}} \braidR_{V^*, V}^{-1}, & V \otimes V^* \\
0, & V \otimes V, V^* \otimes V^*
\end{cases}.
\]
\end{proposition}

\begin{proof}
The algebra $\mathrm{Cl}_q(c)$ was defined as the quotient of $T(H)$ by the ideal generated by the subspace $P_c = \{ x - \beta_c(x) : x \in R_H \}$. Recall that $R_H = \mathrm{im}(\id + \sigma)$ by \cref{prop:iso-image} and that it is spanned by the elements $\mathcal{V}_{i j}$, $\mathcal{W}_{i j}$ and $\mathcal{M}_{i j}$.
We want to show that $P_c$ is equal to $\{ x \otimes y + \sigma(x \otimes y) - (x, y)_c : x, y \in H \}$.
It is clear that the latter contains the elements $\mathcal{V}_{i j}$ and $\mathcal{W}_{i j}$.
Next, since $\mathcal{M}_{i j} = (\id + \sigma)(w_i \otimes v_j)$ and $\beta_c(\mathcal{M}_{i j}) = \delta_{i j} c$, we have
\[
\mathcal{M}_{i j} - \beta_c(\mathcal{M}_{i j}) = w_i \otimes v_j + \sigma(w_i \otimes v_j) - (w_i, v_j)_c.
\]
However $\mathcal{M}_{i j}$ can be written in two different ways, since $(\id + \sigma) (V \otimes V^*) = (\id + \sigma) (V^* \otimes V)$, as observed in the proof of \cref{prop:iso-image}. Indeed observe that
\[
\mathcal{M}_{i j} = (\id + q^{1 - \frac{1}{N}} \braidR_{V^*, V}) (w_i \otimes v_j) = (\id + q^{-1 + \frac{1}{N}} \braidR_{V^*, V}^{-1}) q^{1 - \frac{1}{N}} \braidR_{V^*, V} (w_i \otimes v_j),
\]
from which it follows that $\mathcal{M}_{i j} = (\id + \sigma) q^{1 - \frac{1}{N}} \braidR_{V^*, V} (w_i \otimes v_j)$. Then we need to have
\[
(\id + \sigma - (\cdot, \cdot)_c) (w_i \otimes v_j) = (\id + \sigma - (\cdot, \cdot)_c) q^{1 - \frac{1}{N}} \braidR_{V^*, V} (w_i \otimes v_j).
\]
This implies that the bilinear form should satisfy $(w_i, v_j)_c = (\cdot, \cdot)_c \circ q^{1 - \frac{1}{N}} \braidR_{V^*, V} (w_i \otimes v_j)$. But this is satisfied, since on $V \otimes V^*$ we have defined $(\cdot, \cdot)_c = c \langle \cdot, \cdot \rangle \circ q^{-1 + \frac{1}{N}}\braidR_{V^*, V}^{-1}$.
\end{proof}

\begin{remark}
Observe that the bilinear form $(\cdot, \cdot)_c$ is not symmetric, unlike the classical case. On the other hand it satisfies the property $(\cdot, \cdot)_c \circ \sigma = (\cdot, \cdot)_c$, which reduces to the symmetric case in the classical limit, since $\sigma$ reduces to the flip map.
\end{remark}

\begin{remark}
In the classical limit $(\cdot, \cdot)_c$ reduces to a multiple of the Killing form on $\mathfrak{sl}_N$, since $V$ and $V^*$ can be realized as certain Lie subalgebras $\mathfrak{u}_+$ and $\mathfrak{u}_-$ of $\mathfrak{sl}_N$ (see below).
\end{remark}

This presentation shows that $\mathrm{Cl}_q(c)$ essentially coincides with the quantum Clifford algebra defined in \cite{clifford-spinor}, where the starting point is a braiding satisfying the Hecke condition.
It is also a convenient way to show that $\mathrm{Cl}_q(c)$ is a $\Uqsl$-module algebra.

\begin{corollary}
The algebras $\mathrm{Cl}_q(c)$ are $\Uqsl$-module algebras.
\end{corollary}

\begin{proof}
It suffices to show that the subspace of relations is invariant under the action of $\Uqsl$. By \cref{prop:cliff-presentation} this can be written as $\{ x \otimes y + \sigma(x \otimes y) - (x, y)_c : x, y \in H \}$.
It suffices to notice that the bilinear form $(\cdot, \cdot)_c$ is $\Uqsl$-invariant, since it is defined in terms of the equivariant maps $\langle \cdot, \cdot \rangle$ and $\braidR_{V^*, V}^{-1}$. Then we obtain the conclusion.
\end{proof}

Next we show that the algebras $\mathrm{Cl}_q(c)$ for different values of $c$ are isomorphic.

\begin{proposition}
We have $\mathrm{Cl}_q(c) \cong \mathrm{Cl}_q(1)$ as $\Uqsl$-module algebras.
\end{proposition}

\begin{proof}
We can obtain an isomorphism of two non-homogeneous quadratic algebras $Q(V, P)$ and $Q(V^\prime, P^\prime)$  as follows. Suppose we have a vector space isomorphism $f: V \to V^\prime$ such that $T(f) P = P^\prime$, where $T(f) : T(V) \to T(V^\prime)$ is the extension of $f$ to the corresponding tensor algebras. Then $T(f)$ induces an algebra isomorphism. Now consider the map
\[
f(v_i) = c v_i, \quad f(w_i) = w_i, \quad c \in \mathbb{K}^\times.
\]
It is clearly an equivariant automorphism of $H = V \oplus V^*$. Consider the presentation given in \cref{prop:cliff-presentation} and let $P_c = \{ x \otimes y + \sigma(x \otimes y) - (x, y)_c : x, y \in H \}$. Then it is immediate to check that $T(f) P_c = P_1$, using the fact that $(\cdot, \cdot)_c = c (\cdot, \cdot)_1$ and $(V, V)_1 = (V^*, V^*)_1 = 0$. 
\end{proof}

Finally we discuss the connection between $\mathrm{Cl}_q(c)$ and the quantum Clifford algebras introduced in \cite{qclifford}. We refer to this paper for all unexplained material appearing below.

Let us briefly review the definition of these algebras.
Let $\mathfrak{g}$ be a complex simple Lie algebra. We consider a cominuscule parabolic subalgebra and denote by $\mathfrak{l}$ its Levi factor, and by $\mathfrak{u}_\pm$ its nilradical and its opposite.
The $U_q(\mathfrak{l})$-modules $\mathfrak{u}_\pm$ are simple, hence there is a unique $U_q(\mathfrak{l})$-invariant dual pairing $\langle \cdot, \cdot \rangle : \mathfrak{u}_- \otimes \mathfrak{u}_+ \to \mathbb{C}$, up to a scalar.
It can be extended to a dual pairing $\langle \cdot, \cdot \rangle_k : \Lambda_q^k(\mathfrak{u}_-) \otimes \Lambda_q^k(\mathfrak{u}_+) \to \mathbb{C}$ as in \cite[Proposition 3.6]{qclifford}.
The module $\mathfrak{u}_+$ acts on $\Lambda_q(\mathfrak{u}_+)$ by left multiplication, denoted by $\gamma_+$.
We also obtain an action of $\mathfrak{u}_-$ on $\Lambda_q(\mathfrak{u}_+)$ by dualizing right multiplication on $\Lambda_q(\mathfrak{u}_-)$.
We denote this action by $\gamma_-$.
By \cite[Theorem 5.1]{qclifford} the map $\Lambda_q(\mathfrak{u}_-) \otimes \Lambda_q(\mathfrak{u}_+) \to \mathrm{End}_\mathbb{C}(\Lambda_q(\mathfrak{u}_+))$ is an equivariant isomorphism.

Hence the algebra $\mathrm{End}_\mathbb{C}(\Lambda_q(\mathfrak{u}_+))$, together with its factorization in terms of $\gamma_-$ and $\gamma_+$, can be considered a \emph{quantum Clifford algebra}, which we will denote by $\widetilde{\mathrm{Cl}}_q$.
It should be thought of as the Clifford algebra of the complex tangent space $\mathfrak{u}_+ \oplus \mathfrak{u}_-$.

\begin{remark}
The extension of the dual pairing given in \cite[Proposition 3.6]{qclifford} is certainly not unique. We have at least the choice of a scalar in each degree. As a matter of fact, it is necessary to make such a choice of scalars to recover the relations of the classical Clifford algebra, as discussed in \cite{mat-proj}. For this reason, even though we will simply denote such an algebra by $\widetilde{\mathrm{Cl}}_q$, the choice of such scalars is understood.
\end{remark}

\begin{proposition}
The PBW-deformation $\mathrm{Cl}_q(c)$ is isomorphic to the quantum Clifford algebra $\widetilde{\mathrm{Cl}}_q$, with an appropriate choice of scalars for the latter.
\end{proposition}

\begin{proof}
Let $\{ e_i \}$ and $\{ f_i \}$ be dual bases of $\mathfrak{u}_+$ and $\mathfrak{u}_-$ with respect to the dual pairing $\langle \cdot, \cdot \rangle$.
Moreover we assume that $\{ e_i \}$ is an orthonormal basis with respect to an invariant Hermitian inner product.
We write $\mathfrak{e}_i = \gamma_+(e_i)$ and $\mathfrak{i}_i = \gamma_-(f_i)$ as in \cite{mat-proj}.
By definition these operators satisfy the relations of $\Lambda_q(\mathfrak{u}_+)$ and $\Lambda_q(\mathfrak{u}_-)$, respectively.
Their cross-relations are computed in \cite[Proposition 6.5]{mat-proj} and \cite[Proposition 6.6]{mat-proj}, and are given by
\[
\mathfrak{i}_i \mathfrak{e}_j + q^{1 - \delta_{i j}} \mathfrak{e}_j \mathfrak{i}_i - \delta_{i j} q (q - q^{-1}) \sum_{k = 1}^{i - 1} \mathfrak{e}_k \mathfrak{i}_k = \delta_{i j} \id.
\]
We remark that the bases $\{ v_i \}$ and $\{ w_i \}$ used in this paper do not coincide exactly with $\{ e_i \}$ and $\{ f_i \}$, but they are related by a rescaling of the form $e_i = c_i v_i$ and $f_i = c_i^{-1} w_i$.
However it is immediate to check that these relations remain the same under such a rescaling.

The relations above should be compared to the expression for $\mathcal{M}_{i j}$ given in \eqref{eq:Mij-expression}, with $\lambda^\prime = q$. We see that we obtain exactly the relations of $\mathrm{Cl}_q(1)$, that is with parameter $c = 1$.
The relations for $c \neq 1$ can also be obtained in this way, namely by rescaling the dual pairing $\langle \cdot, \cdot \rangle_k : \Lambda_q^k(\mathfrak{u}_-) \otimes \Lambda_q^k(\mathfrak{u}_+) \to \mathbb{C}$, as explained in \cite{mat-proj}.
\end{proof}

\vspace{3mm}

{\footnotesize
\emph{Acknowledgements}.
I would like to thank Cristian Vay for useful discussions.
}

\bigskip


\begin{thebibliography}{50}

\bibitem[BCDRV96]{clifford-spinor}
R. Bautista, A. Criscuolo, M. Durdević, M. Rosenbaum, J.D. Vergara,
\textit{Quantum Clifford algebras from spinor representations},
Journal of Mathematical Physics 37, no. 11 (1996): 5747-5775.

\bibitem[BeZw08]{bezw}
A. Berenstein, S. Zwicknagl,
\textit{Braided symmetric and exterior algebras},
Transactions of the American Mathematical Society 360, no. 7 (2008): 3429-3472.

\bibitem[BrGa96]{pbw-deformation}
A. Braverman, D. Gaitsgory,
\textit{Poincaré–Birkhoff–Witt theorem for quadratic algebras of Koszul type},
Journal of Algebra 181, no. 2 (1996): 315-328.

\bibitem[Con95]{con-book}
A. Connes,
\textit{Noncommutative geometry}, Academic press, 1995.

\bibitem[ČSV95]{twisted-tensor}
A. Čap, H. Schichl, J. Van\v{z}ura,
\textit{On twisted tensor products of algebras},
Communications in algebra 23, no. 12 (1995): 4701-4735.

\bibitem[DADą10]{dd-proj}
F. D’Andrea, L. Dąbrowski,
\textit{Dirac operators on quantum projective spaces},
Communications in Mathematical Physics 295, no. 3 (2010): 731-790.

\bibitem[Jan96]{jantzen}
J.C. Jantzen
\textit{Lectures on quantum groups},
Vol. 6. American Mathematical Society, 1996.

\bibitem[KlSc97]{klsc}
A.U. Klimyk, K. Schmüdgen,
\textit{Quantum groups and
their representations}, Vol. 552. Berlin: Springer, 1997.

\bibitem[Krä04]{qflag}
U. Krähmer,
\textit{Dirac operators on quantum flag manifolds},
Letters in Mathematical Physics 67, no. 1 (2004): 49-59.

\bibitem[KrTu13]{qclifford}
U. Krähmer, M. Tucker-Simmons,
\textit{On the Dolbeault--Dirac Operator of Quantized Symmetric Spaces},
Transactions of the London Mathematical Society 2, no. 1 (2015): 33-56.

\bibitem[LZZ11]{invariant-theory}
G.I. Lehrer, Hechun Zhang, R.B. Zhang,
\textit{A quantum analogue of the first fundamental theorem of classical invariant theory},
Communications in Mathematical Physics 301, no. 1 (2011): 131-174.

\bibitem[Mat17]{mat-square}
M. Matassa,
\textit{Dolbeault–Dirac Operators, Quantum Clifford Algebras and the Parthasarathy Formula},
Advances in Applied Clifford Algebras 27, no. 2 (2017): 1581-1609.

\bibitem[Mat18a]{mat-proj}
M. Matassa,
\textit{Dolbeault-Dirac operators on quantum projective spaces},
Journal of Lie Theory 28, no. 1 (2018): 211-244.

\bibitem[Mat18b]{mat-lagr}
M. Matassa,
\textit{The Parthasarathy formula and a spectral triple for the quantum Lagrangian Grassmannian of rank two},
preprint at \href{https://arxiv.org/abs/1810.06456}{arXiv:1810.06456} (2018).

\bibitem[PoPo05]{quadratic-algebras}
A. Polishchuk, L. Positselski,
\textit{Quadratic algebras},
Vol. 37. American Mathematical Soc., 2005.

\bibitem[Tuc13]{tucker}
M. Tucker-Simmons,
\textit{Quantum algebras associated to irreducible generalized flag manifolds},
PhD thesis, available at \href{https://arxiv.org/abs/1308.4185}{arXiv:1308.4185} (2013).

\bibitem[WaWi18]{pbw-products}
C. Walton, S. Witherspoon,
\textit{PBW deformations of braided products},
Journal of Algebra 504 (2018): 536-567.

\bibitem[Zwi09]{zwi}
S. Zwicknagl,
\textit{R-matrix Poisson algebras and their deformations},
Advances in Mathematics 220, no. 1 (2009): 1-58.

\end{thebibliography}
\end{document}